\documentclass{birkjour}
\usepackage{amsmath,amssymb,amsthm,color}
 \newtheorem{thm}{Theorem}[section]
 \newtheorem{cor}[thm]{Corollary}
 \newtheorem{lem}[thm]{Lemma}
 
 \theoremstyle{definition}
 \newtheorem{defn}[thm]{Definition}
 \theoremstyle{remark}
 \newtheorem{rem}[thm]{Remark}
 \newtheorem*{ex}{Example}
 \numberwithin{equation}{section}

\def\bc{\begin{center}}
      \def\ec{\end{center}}

 \def\bel{\begin{equation}}
 \def\enl{\end{equation}}
 \def\be{\begin{eqnarray*}}
 \def\en{\end{eqnarray*}}
 \def\R{{\bf R}}

\def\i{{\bf i}}
\def\j{{\bf j}}
\def\k{{\bf k}}
 \def\H{\mathcal{H}}
 \def\N{{\bf N}}

 \def\uq{\underline{q}}

\def\up{\underline{p}}

\renewcommand{\theequation}{\arabic{section}.\arabic{equation}}
\def\i{{\bf i}}

 \def\en{{\bf e}_n}
 
 \def\j{{\bf j}}
 \def\k{{\bf k}}

 \def\uq{\underline{q}}

\usepackage{amssymb}
 \usepackage{amsfonts}
 \usepackage{color}
 \newcommand{\argmin}{\mathop{{\rm arg}\min}}
\begin{document}

%-------------------------------------------------------------------------
% editorial commands: to be inserted by the editorial office
%
%\firstpage{1} \volume{228} \Copyrightyear{2004} \DOI{003-0001}
%\seriesextraline{This is the Concrete Title of this Book\br H.E. R and S.T.C. W, Eds.}
%\firstpage{1}

\title[Discrete Uncertainty Principle in Quaternion Setting]
 {Discrete Uncertainty Principle in Quaternion Setting and Application in Signal Reconstruction}

%----------Author 1
\author{Yan Yang}
\address{School
of Mathematics(Zhuhai), Sun Yat-Sen University,\\
 Zhuhai, 519082, Guangzhou, P.R.China.}

\email{mathyy@sina.com}

%\thanks{This work was completed with the support of Science and Technology Program of Guangzhou, China (No. 74120-42050001), the National Natural Science Foundation of China under Grant (No. 61806027), Macao Science and Technology Development Fund (FDCT/031/2016/A1 and FDCT/085/2018/A2)}

%----------Author 2
\author{Kit~Ian Kou*}
\address{Department of Mathematics, Faculty of Science and Technology,\\
University of Macau, Taipa, Macao, China.}
\email{kikou@umac.mo}
\thanks{*Corresponding author, kikou@umac.mo}

%----------Author 3
\author{Cuiming Zou}\address{School of Information Science and Engineering,\\
Chengdu University, Chengdu, 610106, China.}
\email{zoucuiming2006@163.com}



%\date{January 30, 2019}
%\dedicatory{To my son}
%%% ----------------------------------------------------------------------

\begin{abstract}
In this paper, the uncertainty principle of discrete signals associated with Quaternion Fourier transform is investigated. It suggests how sparsity helps in the recovery of missing frequency.
\end{abstract}

%%% ----------------------------------------------------------------------
\maketitle
%%% ----------------------------------------------------------------------
%\tableofcontents
\section{Introduction}
The classical uncertainty principle (the continuous-time uncertainty principle)
says that if a function $f(t)$ is essentially zero outside an interval of length $\Delta_t$ and its Fourier transform $\hat{f}(\omega)$ defined by
$$\hat{f}(\omega)=\int_{-\infty}^{\infty}f(t)e^{-2\pi\i t\omega}dt$$
is essentially zero outside an interval of length $\Delta_{\omega}$, then
$$\Delta_t\Delta_{\omega}\geq 1.$$
In mathematics, that means a function and its Fourier transform cannot both be higher concentrated.
Uncertainty principle was first introduced by Werner Heisenberg in quantum mechanics \cite{H1927}, which plays an important role in physics and engineering over the past century. Recently, uncertainty principle was applied to signal processing by Donoho and Stark \cite{DS1989}, Candes and Tao \cite{CT2006}, Tropp \cite{T2008}, and Bandeira, Lewis, and Mixon \cite{BLM2017}.
In \cite{DS1989}, the authors first gave the uncertainty principles of discrete 1D signals. It states that: if $\{x_t\}_{t=0}^{N-1}$ is a sequence of length $N$ and $\{\hat{x}_{\omega}\}_{\omega=0}^{N-1}$ is the sequence of its discrete Fourier transform, which is defined by $$\hat{x}_{\omega}:=\frac{1}{\sqrt{N}}\sum_{t=0}^{N-1}x(t)e^{-2\pi\i \frac{\omega t}{N}}, \quad \mbox{ for } \omega=0, 1, \cdots, N-1.$$
Then
\begin{eqnarray}\label{yk5eq111}
N_t {N}_{\omega}\geq N,
\end{eqnarray}
where $\{x_t\}$ is nonzero at $N_t$ points and $\{\hat{x}_{\omega}\}$ is nonzero at ${N}_{\omega}$ points.
%Formula (\ref{yk5eq111}) was applied to signal recovery problem.
In \cite{CT2006}, the uniform uncertainty principle was obtained and which played an crucial role in compressed sensing. The discrete uncertainty principles with applications on sparse signal processing was investigated in \cite{BLM2017}. To the authors' knowledge, the higher dimensional investigation was first considered in \cite{K2017}, inspired by Donoho and Startk's uncertainty principle \cite{DS1989}, the authors \cite{K2017} study the uncertainty principle and signal recovery for continuous quaternion-valued signals.

The quaternion Fourier transform (QFT) plays a vital role in the
representation of 2D signals. It is an extension of Fourier
transform (FT) to the quaternion algebra, which was first proposed by Ell
\cite{E1993}. It transforms a real (or
quaternionic) 2D signal into a quaternion-valued frequency domain
signal. The four components of the QFT separate four cases of
symmetry into real signals instead of only two as in the complex FT. The QFT has wide range of application, such as color image analysis \cite{B1999, SE2007}, color image digital watermarking scheme \cite{BLC2003}, image pre-processing and neural
computing techniques for speech recognition \cite{BTN2007}, envelope \cite{KLMZ2017} and edge detectors \cite{HK2018} of color images.

In this paper, we study a novel discrete uncertainty principle associated with the QFT
and discuss its application to signal recovery. The main contributions of this paper are summarized as follows.

\begin{itemize}

\item[1.] The discrete case of uncertainty principle associated with Quaternion Fourier transform is established to give the relationship between the nonzero numbers of the discrete quaternion-valued signals and their QFTs.

\item[2.] The discrete uncertainty principle suggests how sparsity helps in the recovery of missing frequencies.
\end{itemize}

The article is organized as follows. The Quaternion algebra and
Quaternion Fourier transform are reviewed in
Section \ref{S2}. The uncertainty principle of discrete 2D signals is
obtained for two-sided discrete Quaternion Fourier transform in Section \ref{S3}.
In Section \ref{S4}, the discussion for application of uncertainty principles in spare signal recovery is investigated.
%%%%%%%%%%%%%%%%%%%%%%%%%%%%%%%%%%%%%%%%%%%%%%%%%%%%%%%%%%%%%%%%%%%%%%%%%
\section{Preliminaries}\label{S2}

The quaternion algebra $\mathcal{H}$ was first invented by W.
R. Hamilton in 1843 for extending complex numbers to a 4D algebra
\cite{S1979}. A quaternion $q\in \mathcal{H}$ can be written in this
form
$$q=q_0+\uq=q_0+\i q_1+\j q_2+\k q_3, \; q_k\in \R, \; k=0, 1, 2, 3,$$
where $\i, \j, \k$ satisfy Hamilton's multiplication rules
$$\i^2=\j^2=\k^2=-1, \i\j=-\j\i=\k, $$
$$\j\k=-\k\j=\i, \k\i=-\i\k=\j.$$

Applying the Hamilton's multiplication rules, the multiplication of two
quaternions $p=p_0+\up$ and $q=q_0+\uq$ can be expressed by
$$pq=p_0q_0+\up\cdot\uq+p_0\uq+q_0 \up+\up\times\uq,$$
where $$\up\cdot\uq :=-(p_1q_1+p_2q_2+p_3q_3)$$ and
$$\up\times\uq :=\i(p_3q_2-p_2q_3)+\j(p_1q_3-p_3q_1)+\k(p_2q_1-p_1q_2).$$

We define the conjugation of $q\in \mathcal{H}$ by $\overline{q} :=q_0-\i
q_1-\j q_2-\k q_3$. Clearly,
$q\bar{q}=q_0^2+q_1^2+q_2^2+q_3^2.$ The modulus of a quaternion
$q$ is defined by
$$|q| :=\sqrt{q\bar{q}}=\sqrt{q_0^2+q_1^2+q_2^2+q_3^2}.$$

In this paper, we study the quaternion-valued signal $f:
\R^2\rightarrow \H$ which can be expressed by
$$f(t, s)=f_0(t, s)+\i f_1(t, s)+\j f_2(t, s)+\k f_3(t, s),$$ where $f_k, (k=0, 1, 2, 3)$ are real-valued functions. %Especially, for real-valued signal $f$, the results also holds.

In 1997, Sangwine \cite{S1997} defined the fundamental idea of a discrete Quaternion Fourier transform (DQFT) and inversion discrete Quaternion Fourier transform (IDQFT) of Quaternion-valued signals, which we recall next.
%Firstly, we give the definition of two-sided discrete Quaternion Fourier transform.

%======================================
\begin{defn}\label{Def2.1} {\bf (DQFT and IDQFT)}
Let $\{f(t,s)\}$ $(t=0, 1, \cdots, M-1, s=0, 1, \cdots, N-1)$ be a sequence of length $MN$ and $\{\hat{f}(u, v)\}$ $(u=0, 1, \cdots, M-1, v=0, 1, \cdots, N-1)$ be its two-sided discrete Quaternion Fourier transform (DQFT), which is defined by
\begin{equation}\label{yk5eq1}
\hat{f}(u, v) :=\frac{1}{\sqrt{MN}}\sum_{t=0}^{M-1}\sum_{s=0}^{N-1}e^{-2\pi\i\frac{ ut}{M}}f(t, s)e^{-2\pi\j\frac{ vs}{N}}.
\end{equation}

Moreover, the inverse discrete Quaternion Fourier transform (IDQFT) of $\{\hat{f}(u, v)\}$ is defined by
\begin{equation}\label{yk5eq2}
f(t, s) :=\frac{1}{\sqrt{MN}}\sum_{u=0}^{M-1}\sum_{v=0}^{N-1}e^{2\pi\i\frac{ ut}{M}}\hat{f}(u, v)e^{2\pi\j\frac{ vs}{N}}.
\end{equation}
\end{defn}
%----------------------

As a consequence of Definition \ref{Def2.1}, formula (\ref{yk5eq1}) can be represented in the matrix form. Denote two $M \times N$ matrices
\begin{eqnarray*}
A %& %:=&%\left(f(t, s)\right)_N\\
&:=&\left(
\begin{array}{cccc}
f(0, 0) & f(0, 1) & \cdots & f(0, N-1)\\
f(1, 0) & f(1, 1) & \cdots & f(1, N-1)\\
\vdots  & \vdots  & \vdots & \vdots   \\
f(M-1, 0) & f(M-1, 1) & \cdots & f(M-1, N-1)
\end{array}
\right)\\
&=&\left( f(t,s) \right) \in \mathcal{H}^{M \times N}
\end{eqnarray*}
and
\begin{eqnarray*}
\hat{A}%&=&\left(\hat{f}(u, v)\right)_N\\
&:=&\left(
\begin{array}{cccc}
\hat{f}(0, 0) & \hat{f}(0, 1) & \cdots & \hat{f}(0, N-1)\\
\hat{f}(1, 0) & \hat{f}(1, 1) & \cdots & \hat{f}(1, N-1)\\
\vdots  & \vdots  & \vdots & \vdots   \\
\hat{f}(M-1, 0) & \hat{f}(M-1, 1) & \cdots & \hat{f}(M-1, N-1)
\end{array}
\right)\\
&=&\left(\hat{f}(u,v) \right) \in \mathcal{H}^{M \times N},
\end{eqnarray*}
then formula (\ref{yk5eq1}) can be expressed as
$$\hat{A}=V_{\i}AV_{\j},$$
where $V_{\i}$ and $V_{\j}$ are $M \times M$ and $N \times N$ Vandermonde matrices, which are defined by
\begin{eqnarray*}
V_{\i} :=\frac{1}{\sqrt{M}}\left(
\begin{array}{cccc}
1 & 1 & \cdots & 1\\
1 & e^{-2\pi\i\frac{1}{M}} & \cdots & e^{-2\pi\i\frac{ (M-1)}{M}}\\
\vdots  & \vdots  & \vdots & \vdots   \\
1 & e^{-2\pi\i\frac{ (M-1)}{M}}& \cdots & e^{-2\pi\i\frac{(M-1)^2}{M}}
\end{array}
\right)
\end{eqnarray*}
and
\begin{eqnarray*}
V_{\j} :=\frac{1}{\sqrt{N}}\left(
\begin{array}{cccc}
1 & 1 & \cdots & 1\\
1 & e^{-2\pi\j\frac{1}{N}} & \cdots & e^{-2\pi\j\frac{ (N-1)}{N}}\\
\vdots  & \vdots  & \vdots & \vdots   \\
1 & e^{-2\pi\j\frac{ (N-1)}{N}}& \cdots & e^{-2\pi\j\frac{(N-1)^2}{N}}
\end{array}
\right),
\end{eqnarray*} respectively.
Clearly, they are non-singular matrices.

Similarly, formula (\ref{yk5eq2}) can be expressed as
\begin{equation}\label{yk5eq3}
A=V_{\i}^{-1}\hat{A}V_{\j}^{-1}=V_{-\i}\hat{A}V_{-\j}.
\end{equation}

%==============================================================
\section{The Discrete Uncertainty Principle in Quaternion Setting}\label{S3}
Uncertainty principle has a significant role in both science and engineering for most
of the past century. In this section, we show that the discrete uncertainty principle of quaternion-valued signals.
%--------------------------------
\begin{thm}[Main Theorem] \label{yk5th1}
Let $N_{(t, s)}$ and $N_{(u, v)}$ be the numbers of nonzero elements of sequences $\{f(t, s)\}$ $(t=0, 1, \cdots, M-1, s=0, 1, \cdots, N-1)$ and $\{\hat{f}(u, v)\}$ $(u=0, 1, \cdots, M-1, v=0, 1, \cdots, N-1)$, respectively. Then we have
\begin{equation}\label{yk5eq3.5}
N_{(t, s)}\cdot N_{(u, v)}\geq MN.
\end{equation}
\end{thm}

By the arithmetic mean-Geometric mean inequality, which we describe next.
%=============================================================
\begin{cor}\label{yk5th2}
$$N_{(t, s)} + N_{(u, v)}\geq 2\sqrt{MN}.$$
\end{cor}

In particular, when $M=N$, we have
%=====================================
\begin{cor}
%Let $N_{(t, s)}$ and $N_{(u, v)}$ be the numbers of nonzero elements of $\{f(t, s)\}, \mbox{ }(t=0, s=0, 1, \cdots, N-1)$ and $\hat{f}(u, v)\}, \mbox{ }(u, v=0, 1, \cdots, N-1)$, respectively. Then we have
\begin{equation}
N_{(t, s)}\cdot N_{(u, v)}\geq N^{2}.
\end{equation}
\end{cor}

%=====================
\begin{cor}[1D discrete uncertainty principle]
In Theorem \ref{yk5th1}, when $M=1$ or $N=1$, one has the classical discrete uncertainty principle (\ref{yk5eq111}) in \cite{DS1989}.
\end{cor}

\begin{rem}
Theorem \ref{yk5th1} gives a lower bound on the value of the time-bandwidth product. Corollary \ref{yk5th2} gives a lower bound of the sum of nonzero elements in both time and frequency spaces. Furthermore, it is easy to construct examples on the equality of (\ref{yk5eq3.5}) in Theorem \ref{yk5th1}. In this sense, the discrete-time principle is sharp.

\end{rem}

%======================================
\begin{ex}[A trivial case] \label{yk5exa1}
Let $t=0, 1, \cdots, M-1, s=0, 1, \cdots, N-1$ and
\begin{eqnarray*}
f(t, s)=\left\{
\begin{array}{lll}
1, & \mbox{ for } t=0, s=0,\\
0,& \mbox{otherwise}.
\end{array}
\right.
\end{eqnarray*}
As a consequence of (\ref{yk5eq1}), we have
$$\hat{f}(u, v)=\frac{1}{\sqrt{MN}}f(0, 0)=\frac{1}{\sqrt{MN}},$$  when $u=0, 1, \cdots, M-1$ and $v=0, 1, \cdots, N-1.$
Therefore, $N_{(t,s)}=1$, $N_{(u, v)}=MN$,  or equivalently $N_{(t,s)}\cdot N_{(u, v)}=MN$ as desired.
\end{ex}

The next result is a non-trivial example.
%=================================================================
\begin{ex}
Suppose that $M=N$ admits the factorization $N=k\cdot l$ and
\begin{eqnarray}\label{yk5eq6}
A =\left( f(t, s) \right) = \left(
\begin{array}{cccc}
E_k & E_k & \cdots & E_k\\
E_k & E_k & \cdots & E_k\\
\vdots  & \vdots  & \vdots & \vdots   \\
E_k & E_k& \cdots & E_k
\end{array}
\right),
\end{eqnarray}
where
\begin{eqnarray*}
E_k :=\left(
\begin{array}{cccc}
1 & 0 & \cdots & 0\\
0 & 0 & \cdots & 0\\
\vdots  & \vdots  & \vdots & \vdots   \\
0 & 0& \cdots & 0
\end{array}
\right)
\end{eqnarray*} is a $k \times k$ matrix with entries are determined by $e_{st} = 1$ if $s=t=1$, otherwise $0$.
We can prove that
\begin{eqnarray*}
\hat{A}=\left(\hat{f}(u, v)\right)=\frac{l}{k}\left(
\begin{array}{cccc}
E_l & E_l & \cdots & E_l\\
E_l & E_l & \cdots & E_l\\
\vdots  & \vdots  & \vdots & \vdots   \\
E_l & E_l& \cdots & E_l
\end{array}
\right)
\end{eqnarray*}
is a $N \times N$ element block matrix with $k \times k$ sub-Matrices $E_l$.
Here,
$
E_l=\left(
\begin{array}{cccc}
1 & 0 & \cdots & 0\\
0 & 0 & \cdots & 0\\
\vdots  & \vdots  & \vdots & \vdots   \\
0 & 0& \cdots & 0
\end{array}
\right)
$
with entries are determined by $e_{u,v} :=1$ if $u=v=1$, otherwise is 0. Equivalently, we have $N_{(t, s)}=l^2$ and $N_{(u, v)}=k^2$.
Therefore, $$N_{(t, s)}\cdot N_{(u, v)}=l^2\cdot k^2=N^2. $$
\end{ex}
%========================================================================
\begin{proof} Let $p, q=0, 1, \cdots, l-1$. As a consequence of (\ref{yk5eq6}), we have
\begin{eqnarray}\label{yk5eq7}
f(t, s)=\left\{
\begin{array}{lll}
1, & \mbox{ for } t=kp, s=kq,\\
0,& \mbox{otherwise}.
\end{array}
\right.
\end{eqnarray}
That is
\begin{eqnarray}\label{yk5eq8}
A&=&\left(
\begin{array}{cccc}
f(0, 0) & f(0, k) & \cdots & f(0, (l-1)k)\\
f(k, 0) & f(k, k) & \cdots & f(k, (l-1)k)\\
\vdots  & \vdots  & \vdots & \vdots   \\
f((l-1)k, 0) & f((l-1)k, k)& \cdots & f((l-1)k, (l-1)k)
\end{array}
\right)\nonumber\\
&=&
\left(
\begin{array}{cccc}
1 & 1 & \cdots & 1\\
1 & 1 & \cdots & 1\\
\vdots  & \vdots  & \vdots & \vdots   \\
1 & 1& \cdots & 1
\end{array}
\right).
\end{eqnarray}
Then we have to prove that
\begin{eqnarray*}
\hat{f}(u, v)=\left\{
\begin{array}{lll}
\frac{l}{k}, & \mbox{ for } u=al, v=bl,\\
0,& \mbox{otherwise},
\end{array}
\right.
\end{eqnarray*}
where $a, b=0, 1, \cdots, k-1.$

In fact, for $a, b=0, 1, \cdots, k-1$, $N=l \cdot k$ and as a consequence of (\ref{yk5eq1}), (\ref{yk5eq7}) and (\ref{yk5eq8}), we have

\begin{eqnarray}\label{yk5eq9}
\hat{f}(al, bl)&=&\frac{1}{N}\sum_{t=0}^{N-1}\sum_{s=0}^{N-1}e^{-2\pi\i \frac{tal}{N}}f(t, s)e^{-2\pi\j\frac{sbl}{N}}\nonumber\\
%&=&\frac{1}{kl}\sum_{t=0}^{kl-1}\sum_{s=0}^{kl-1}e^{-2\pi\i \frac{tal}{kl}}f(t, s)e^{-2\pi\j\frac{sbl}{kl}}\nonumber\\
&=&\frac{1}{kl}\sum_{t=0}^{kl-1}\sum_{s=0}^{kl-1}e^{-2\pi\i \frac{ta}{k}}f(t, s)e^{-2\pi\j\frac{sb}{k}}\nonumber\\
&=&\frac{1}{kl}\sum_{p=0}^{l-1}\sum_{s=0}^{l-1}e^{-2\pi\i \frac{kpa}{k}}e^{-2\pi\j\frac{kqb}{k}}.
\end{eqnarray}
Therefore, we obtain
\begin{eqnarray}
\hat{f}(al, bl)&=&\frac{1}{kl}\sum_{p=0}^{l-1}\sum_{s=0}^{l-1}e^{-2\pi\i pa}e^{-2\pi\j qb}\nonumber\\
&=&\frac{1}{kl}\sum_{p=0}^{l-1}\sum_{s=0}^{l-1}1=\frac{l}{k}\nonumber.
\end{eqnarray}
When $u\neq al, v\neq bl$, for $a, b=0, 1, \cdots, k-1$, we have
\begin{eqnarray*}
\hat{f}(u, v)=\frac{1}{kl}\sum_{t=0}^{kl-1}\sum_{s=0}^{kl-1}e^{-2\pi\i\frac{ut}{kl}}f(t, s)e^{-2\pi\j\frac{vs}{kl}}.
\end{eqnarray*}
As a consequence of (\ref{yk5eq7}) and (\ref{yk5eq8}), we obtain

\begin{eqnarray*}
&&\hat{f}(u, v)\\
&=&f(0, 0)+f(0, k)e^{-2\pi\j\frac{kv}{kl}}+\cdots+f(0, (l-1)k)e^{-2\pi\j\frac{kv(l-1)}{kl}}\\
&&+e^{-2\pi\i\frac{ku}{kl}}[f(k, 0)+f(k, k)e^{-2\pi\j\frac{kv}{kl}}+\cdots+f(k, (l-1)k)e^{-2\pi\j\frac{kv(l-1)}{kl}}]\\
&&+\cdots\\
&&+e^{-2\pi\i\frac{(l-1)ku}{kl}}[f(k, 0)+f((l-1)k, k)e^{-2\pi\j\frac{kv}{kl}}\\
&&+\cdots+f((l-1)k, (l-1)k)e^{-2\pi\j\frac{kv(l-1)}{kl}}]\\
&=&1+e^{-2\pi\j\frac{v}{l}}+\cdots+e^{-2\pi\j\frac{v(l-1)}{l}}\\
&&+e^{-2\pi\i\frac{u}{l}}[1+e^{-2\pi\j\frac{v}{l}}+\cdots+e^{-2\pi\j\frac{v(l-1)}{l}}]\\
&&+\cdots\\
&&+e^{-2\pi\i\frac{(l-1)u}{l}}[1+e^{-2\pi\j\frac{v}{l}}+\cdots+e^{-2\pi\j\frac{v(l-1)}{lh}}]=0.
\end{eqnarray*}
This yields the desired conclusion.
\end{proof}
%==========================================================================================
To show the discrete uncertainty principle for quaternion-valued signal, the consecutive $m \times n$ sub-matric stated in the following definition are sufficient.
%To show the discrete uncertainty principle for quaternion-valued signal, we shall discuss the consecutive sub-matric of a given matrix.
\begin{defn}[Consecutive Sub-Matric of a Given Matrix]
Given a $M \times N$ matrix \begin{eqnarray}
A=\left(
\begin{array}{cccc}
a_{1,1} & a_{1,2} & \cdots & a_{1,N}\\
\vdots & &\ddots &\vdots \\
a_{M,1}  & a_{M,2}  & \ldots & a_{M,N} \nonumber
\end{array}
\right),
\end{eqnarray}
define the $(2M-1) \times (2N-1)$ matrix ${\Lambda_A}$ as follows
\begin{eqnarray}
{\Lambda_A}:=\left(
\begin{array}{ccccccc}
a_{1,1} & a_{1,2} & \cdots & a_{1,N} & a_{1,1} &\cdot & a_{1,N-1}\\
\vdots &\vdots &\ddots &\vdots & \vdots &\ddots &\vdots\\
a_{M,1}  & a_{M,2}  & \ldots & a_{M,N} &a_{M,1} &\ldots & a_{M, N-1}\\
a_{1,1} & a_{1,2} & \cdots & a_{1,N} &   &  &  \\
\vdots&\vdots &\ddots & \vdots & &0 & \\
a_{M-1,1}  & a_{M,2}  & \ldots & a_{M-1,N} &  & & \nonumber
\end{array}
\right),\end{eqnarray} then the consecutive $m \times n$ sub-matric of $A$ are defined by the $m \times n$ sub-matric of $\Lambda_A$ with rows in $m$ consecutive terms or columns in $n$ consecutive terms, where integers $m \leq M$ and $n \leq N$.

\end{defn}
%=================================================================================================

The following lemmas will be essential in proving these discrete uncertainty principle. 
%To proceed, we need to give the following Lemmas. % Theorem \ref{yk5th1}, we give the Lemmas as follows.
Let $[r]$ be the smallest integer greater than or equal to $r$ and denote $m:= \sqrt{N_{(t, s)}}$.

%---------------------------------------
\begin{lem}\label{yk5lem1}
Let $1<N_{(t, s)}\leq MN$. Assume that the sequence $\{f(t, s)\}$ has $N_{(t, s)}$ nonzero elements ($ t=0, 1, \cdots, M-1, s=0, 1, \cdots, N-1$). %Denote $m := \sqrt{N_{(t, s)}}$,
If $[m] = m$ and $m\leq \min\{M, N\}$, then the sequence $\{ \hat{f}(u, v)\}$ ($u=0, 1, \cdots, M-1, v=0, 1, \cdots, N-1$) forms any consecutive $m-$matrix, i.e.,
\begin{eqnarray}\label{yk5eq4}
&&\hat{A}_{u, v}=\left(\hat{f}(u, v)\right)\\
&=&\left(
\begin{array}{cccc}
\hat{f}(u, v) & \hat{f}(u, v+1) & \cdots & \hat{f}(u, v+m-1)\\
\hat{f}(u+1, v) & \hat{f}(u+1, v+1) & \cdots & \hat{f}(u+1, v+m-1)\\
\vdots  & \vdots  & \vdots & \vdots   \\
\hat{f}(u+m-1, v) & \hat{f}(u+m-1, v+1) & \cdots & \hat{f}(u+m-1, v+m-1)\nonumber
\end{array}
\right),
\end{eqnarray}
which has at least one nonzero element.
\end{lem}

%==================================
\begin{proof}Without loss of generality, let $f(s,t)\neq 0$ ($s, t=0, 1, \cdots m-1$). Denote $m \times m$ matrix by
\begin{eqnarray*}
A_{0, 0}& :=&\left(f(t, s)\right)\\
&=&\left(
\begin{array}{cccc}
f(0, 0) & f(0, 1) & \cdots & f(0, m-1)\\
f(1, 0) & f(1, 1) & \cdots & f(1, m-1)\\
\vdots  & \vdots  & \vdots & \vdots   \\
f(m-1, 0) & f(m-1, 1) & \cdots & f(m-1, m-1).
\end{array}
\right),
\end{eqnarray*}
then $A_{0, 0}\neq {\mathbf 0}$. Here ${\mathbf 0}$ is a $m \times m$ zero matrix.

Then one only needs to prove that any consecutive $m-$matix $\hat{A}_{u, v} \neq {\mathbf 0}$. We prove it by contradiction.
Suppose that there exists $\hat{A}_{u, v}= {\mathbf 0}$, applying (\ref{yk5eq3}), we have
$$A_{0, 0}=({V_{\i}^u})^{-1}\hat{A}_{u, v}({V_{\j}^v})^{-1}={\mathbf 0}.$$
Here
\begin{eqnarray*}
V_{\i}^u :=\frac{1}{\sqrt{M}}\left(
\begin{array}{cccc}
1 & 1 & \cdots & 1\\
1 & e^{-2\pi\i\frac{u}{M}} & \cdots & e^{-2\pi\i\frac{ (m-1)u}{M}}\\
\vdots  & \vdots  & \vdots & \vdots   \\
1 & e^{-2\pi\i\frac{ (m-1)u}{M}}& \cdots & e^{-2\pi\i\frac{(m-1)^2u}{M}}
\end{array}
\right)
\end{eqnarray*}
and
\begin{eqnarray*}
V_{\j}^v :=\frac{1}{\sqrt{N}}\left(
\begin{array}{cccc}
1 & 1 & \cdots & 1\\
1 & e^{-2\pi\j\frac{v}{N}} & \cdots & e^{-2\pi\j\frac{ (m-1)v}{N}}\\
\vdots  & \vdots  & \vdots & \vdots   \\
1 & e^{-2\pi\j\frac{ (m-1)v}{N}}& \cdots & e^{-2\pi\j\frac{(m-1)^2v}{N}}
\end{array}
\right),
\end{eqnarray*} respectively.
It contradicts with $A_{0, 0}\neq {\mathbf 0}$. Therefore, $\hat{A}_{u, v}\neq {\mathbf 0}$. This also means that the sequence $\{ \hat{f}(u, v) \}$ has at least one nonzero element. It completes the proof.\end{proof}
%=======================
The following example illustrates the consecutive $m-$matix $\hat{A}_{u, v}$.
\begin{ex} For $M=2$, $N=3$, let
\begin{eqnarray*}
A=\left(
\begin{array}{ccc}
{f}(0, 0) & {f}(0, 1) & {f}(0, 2)\\
{f}(1, 0) & {f}(1, 1) & {f}(1, 2)
\end{array}
\right).
\end{eqnarray*}
We have
\begin{eqnarray*}
\hat{A}=\left(
\begin{array}{ccc}
\hat{f}(0, 0) & \hat{f}(0, 1) & \hat{f}(0, 2)\\
\hat{f}(1, 0) & \hat{f}(1, 1) & \hat{f}(1, 2)
\end{array}
\right).
\end{eqnarray*}
As a consequence of periodic, the consecutive $2-$matices are:
$$\hat{A}_{0, 0}=\left(
\begin{array}{cc}
\hat{f}(0, 0) & \hat{f}(0, 1) \\
\hat{f}(1, 0) & \hat{f}(1, 1)
\end{array}
\right),$$

$$\hat{A}_{0, 1}=\left(
\begin{array}{cc}
\hat{f}(0, 1) & \hat{f}(0, 2) \\
\hat{f}(1, 1) & \hat{f}(1, 2)
\end{array}
\right)$$
and
$$\hat{A}_{0, 2}=\left(
\begin{array}{cc}
\hat{f}(0, 2) & \hat{f}(0, 0) \\
\hat{f}(1, 2) & \hat{f}(0, 1)
\end{array}
\right).$$

\end{ex}

The following corollary is an immediate consequence of Lemma \ref{yk5lem1}.
%==================================
\begin{cor}\label{yk5rem1}
If the sequence $\{f(t, s)\}$ $(t=0, 1, \cdots, M-1, s=0, 1, \cdots, N-1)$ has $MN$ nonzero elements, then the sequence $\{ \hat{f}(u, v)\}$ $(u=0, 1, \cdots, M-1, v=0, 1, \cdots, N-1)$ has at least one nonzero element. Thus $N_{t, s}\cdot N_{u, v}\geq MN.$
\end{cor}

%========================================
\begin{lem}\label{yk5lem2}
Let $1<N_{(t, s)}<MN$. Assume that the sequence $\{ f(t, s)\}$ $(t=0, 1, \cdots, M-1, s=0, 1, \cdots, N-1)$ has $N_{(t, s)}$ nonzero elements. If $[m] \neq m$ and $m\leq\min\{M, N\}$, then the sequence $\{ \hat{f}(u, v)\}$ $(u=0, 1, \cdots, M-1, v=0, 1, \cdots, N-1)$ forms any consecutive $m-$matix (\ref{yk5eq4}) which has at least two nonzero elements.
\end{lem}
%=====================================================
\begin{proof}To prove it by contradiction, let
\begin{eqnarray*}
A_{0, 0}&=&\left(f(t, s)\right)\\
&=&\left(
\begin{array}{cccc}
f(0, 0) & f(0, 1) & \cdots & f(0, m-1)\\
f(1, 0) & f(1, 1) & \cdots & f(1, m-1)\\
\vdots  & \vdots  & \vdots & \vdots   \\
f(m-1, 0) & f(m-1, 1) & \cdots & f(m-1, m-1)
\end{array}
\right).
\end{eqnarray*}
Without loss of generality, assume that the non-zeros number $N_{(t, s)}$ of the sequence $\{f(t, s)\}$ are all in $A$, and if the sequence $\{\hat{f}(u, v)\}$ in some consecutive $m-$matix (\ref{yk5eq4}) only has one nonzero point, with the aid of (\ref{yk5eq3}), we have
\begin{eqnarray*}
A_{0, 0}&=&(V_{\i}^{u})^{-1}\hat{A}_{u, v}(V_{\j}^{v})^{-1}\\
&=&V_{-\i}^u\hat{A}_{u, v}V_{-\j}^{v}\\
&=&\frac{1}{\sqrt{MN}}\left(
\begin{array}{cccc}
\hat{f}(u, v) & \hat{f}(u, v) & \cdots & \hat{f}(u, v)\\
\hat{f}(u, v) & \hat{f}(u, v) & \cdots & \hat{f}(u, v)\\
\vdots  & \vdots  & \vdots & \vdots   \\
\hat{f}(u, v) & \hat{f}(u, v) & \cdots & \hat{f}(u, v)
\end{array}
\right).
\end{eqnarray*}
That means the sequence $\{f(t, s)\}$ has at lease $m^2=[\sqrt{N_{t, s}}]^2 > N_{(t, s)}$ nonzero points. This contradicts with the condition of the sequence $\{f(t, s)\}$ which has $N_{(t, s)}$ nonzero elements. It completes the proof.\end{proof}
%=================================================================
\begin{lem}\label{yk5lem3}
Let $N\in \N^{+}$ be a positive number. Then we  have
\begin{equation}\label{yk5eq5}
[\sqrt{N}]^2\leq 2N.
\end{equation}
\end{lem}
\begin{proof} It is straightforward to verify the inequality (\ref{yk5eq5}) is true for $N=1, 2, 3, 4$ and $5$.
When $N\geq 6$, one obtains
\begin{eqnarray*}
N^2+1>6N &\Leftrightarrow& N^2-2N+1>4N \\
&\Leftrightarrow& (N-1)^2>4N\\
&\Leftrightarrow& N-1>2\sqrt{N}\\
&\Leftrightarrow& N>2\sqrt{N}+1.
\end{eqnarray*}
Therefore, for $N\geq 6$,
$$[\sqrt{N}]^2<(\sqrt{N}+1)^2=N+2\sqrt{N}+1<2N.$$
Consequently, equality (\ref{yk5eq5}) holds. It completes the proof.\end{proof}
%==================================================

Without loss of generality, we assume that $M\leq N$. A similar argument as in the proof of Lemma \ref{yk5lem1} is also true for the following lemma.

%====================================================
\begin{lem}\label{yk5lem13}
Let $1<N_{(t, s)}\leq MN$. Assume that the sequence $\{f(t, s)\}$ has $N_{(t, s)}$ nonzero elements  ($t=0, 1, \cdots, M-1, s=0, 1, \cdots, N-1$).  %Denote $[\sqrt{N_{(t, s)}}]=m$,
When $m>M=\min\{M, N\}$, then the sequence $\{\hat{f}(u, v)\}$ $(u=0, 1, \cdots, M-1, v=0, 1, \cdots, N-1)$ forms any consecutive $M\times (m-1)-$matix
\begin{eqnarray}\label{yk5eq4}
&&\hat{A}_{u, v}=\left(\hat{f}(u, v)\right)\\
&=&\left(
\begin{array}{cccc}
\hat{f}(u, v) & \hat{f}(u, v+1) & \cdots & \hat{f}(u, v+m-2)\\
\hat{f}(u+1, v) & \hat{f}(u+1, v+1) & \cdots & \hat{f}(u+1, v+m-2)\\
\vdots  & \vdots  & \vdots & \vdots   \\
\hat{f}(u+M-1, v) & \hat{f}(u+M-1, v+1) & \cdots & \hat{f}(u+M-1, v+m-2)\nonumber
\end{array}
\right),
\end{eqnarray}
which has at least one nonzero element.
\end{lem}

%======================================
We can now proceed to the proof of Main Theorem \ref{yk5th1}.
%\vspace{1cm}

%===============================================
\begin{proof} [Proof of Main Theorem \ref{yk5th1}]

\begin{itemize}
\item When $N_{(t, s)}=1$ and $N_{(u,v)}=M N$, with the aid of Example \ref{yk5exa1} and Corollary \ref{yk5rem1}, the conclusion holds.

\item When $1 < N_{(t, s)}< MN$. Without loss of generality, we assume that $M\leq N$. There are two cases:
\begin{itemize}
\item[1)]
Assume that $[\sqrt{N_{(t, s)}}]\leq M$, when $[\sqrt{N_{(t, s)}}]=\sqrt{N_{(t, s)}}$,
by Lemma \ref{yk5lem1}, we have
\begin{eqnarray*}
N_{(t, s)}\cdot N_{(u, v)}&\geq& \frac{M}{\sqrt{N_{(t, s)}}}\frac{N}{\sqrt{N_{(t, s)}}} 1 N_{(t, s)}\\
&=&\frac{MN}{(\sqrt{N_{(t, s)}})^2}N_{(t, s)}\\
&=&MN.
\end{eqnarray*}
When $[\sqrt{N_{(t, s)}}]\neq\sqrt{N_{(t, s)}}$, then $[\sqrt{N_{(t, s)}}]>\sqrt{N_{(t, s)}}$. By Lemma \ref{yk5lem2} and (\ref{yk5eq5}), we have
\begin{eqnarray*}
N_{(t, s)}\cdot N_{(u, v)}&\geq& \frac{M}{[\sqrt{N_{(t, s)}}]}\frac{N}{[\sqrt{N_{(t, s)}}]} 2 N_{(t, s)}\\
&\geq&MN.
\end{eqnarray*}

\item[2)] Assume that $[\sqrt{N_{(t, s)}}]>M$, then $\sqrt{N_{(t, s)}}\geq M$.
By Lemma \ref{yk5lem13}, we have
\begin{eqnarray*}
N_{(t, s)}\cdot N_{(u, v)}&\geq& \frac{N}{[\sqrt{N_{(t, s)}}]-1} 1 N_{(t, s)}\\
&\geq& \frac{N}{\sqrt{N_{(t, s)}}} N_{(t, s)}\\
&=&N\sqrt{N_{(t, s)}} \geq M N.
\end{eqnarray*}
\end{itemize}
\end{itemize}
This completes the proof.\end{proof}
%======================================

The following examples illustrate the discrete uncertainty principle.
\begin{ex} For $M=N=2$, let
\begin{eqnarray*}
A=\left(\begin{array}{cc}
{f}(0, 0) & {f}(0, 1)\\
{f}(1, 0) & {f}(1, 1)
\end{array}
\right).
\end{eqnarray*}
We have
\begin{eqnarray*}
\hat{A}&=&\frac{1}{\sqrt{2}}\left(
\begin{array}{cc}
1 & 1\\
1 & e^{-\pi\i}
\end{array}
\right)\left(
\begin{array}{cc}
{f}(0, 0) & {f}(0, 1)\\
{f}(1, 0) & {f}(1, 1)
\end{array}
\right)
\frac{1}{\sqrt{2}}\left(
\begin{array}{cc}
1 & 1\\
1 & e^{-\pi \j}
\end{array}
\right)\\
&=&
\left(
\begin{array}{cc}
\hat{f}(0, 0) & \hat{f}(0, 1)\\
\hat{f}(1, 0) & \hat{f}(1, 1)
\end{array}
\right).
\end{eqnarray*}
Here
\begin{eqnarray*}
\hat{f}(0, 0)&=& \frac{1}{2}\left[f(0, 0)+f(0, 1)+f(1, 0)+f(1, 1)\right]\\
\hat{f}(0, 1)&=& \frac{1}{2}\left[f(0, 0)+f(1, 0)+\left(f(0, 1)+f(1, 1)\right)e^{-\pi \j }\right]\\
\hat{f}(1, 0)&=& \frac{1}{2}\left[f(0, 0)+f(0, 1)+ e^{-\pi \i }\left(f(1, 0)+f(1, 1)\right)\right]\\
\hat{f}(1, 1)&=& \frac{1}{2}\left[f(0, 0)+e^{-\pi\i}f(1, 0)+f(0, 1)e^{-\pi \j }+e^{-\pi \i }f(1, 1)e^{-\pi \j }\right].
\end{eqnarray*}

%=====================================
%Figure \ref{fig:Ex3.4} shows an example for $\{f(t, s)\}$ with $M=N=2$.

%====================================
We conclude that \begin{itemize}
\item
When $\{f(t, s)\}$ has only one nonzero point, clearly, $\{\hat{f}(u, v)\}$ has 4 nonzero points.

\item When $\{f(t, s)\}$ has two or three nonzero points, $\{\hat{f}(u, v)\}$ has at least 2 nonzero points.

\item When $\{f(t, s)\}$ has four nonzero points, then $\{\hat{f}(u, v)\}$ has at lease 1 nonzero point.
\end{itemize}
Therefore we have $N_{(t, s)}\cdot N_{(u, v)}\geq 4.$ It demonstrates the discrete uncertainty principle. Figure \ref{fig:Ex3.4} shows an example for $\{f(t, s)\}$ with $M=N=2$.
\end{ex}

\begin{figure}[!h]
  \centering
   \includegraphics[height=6cm]{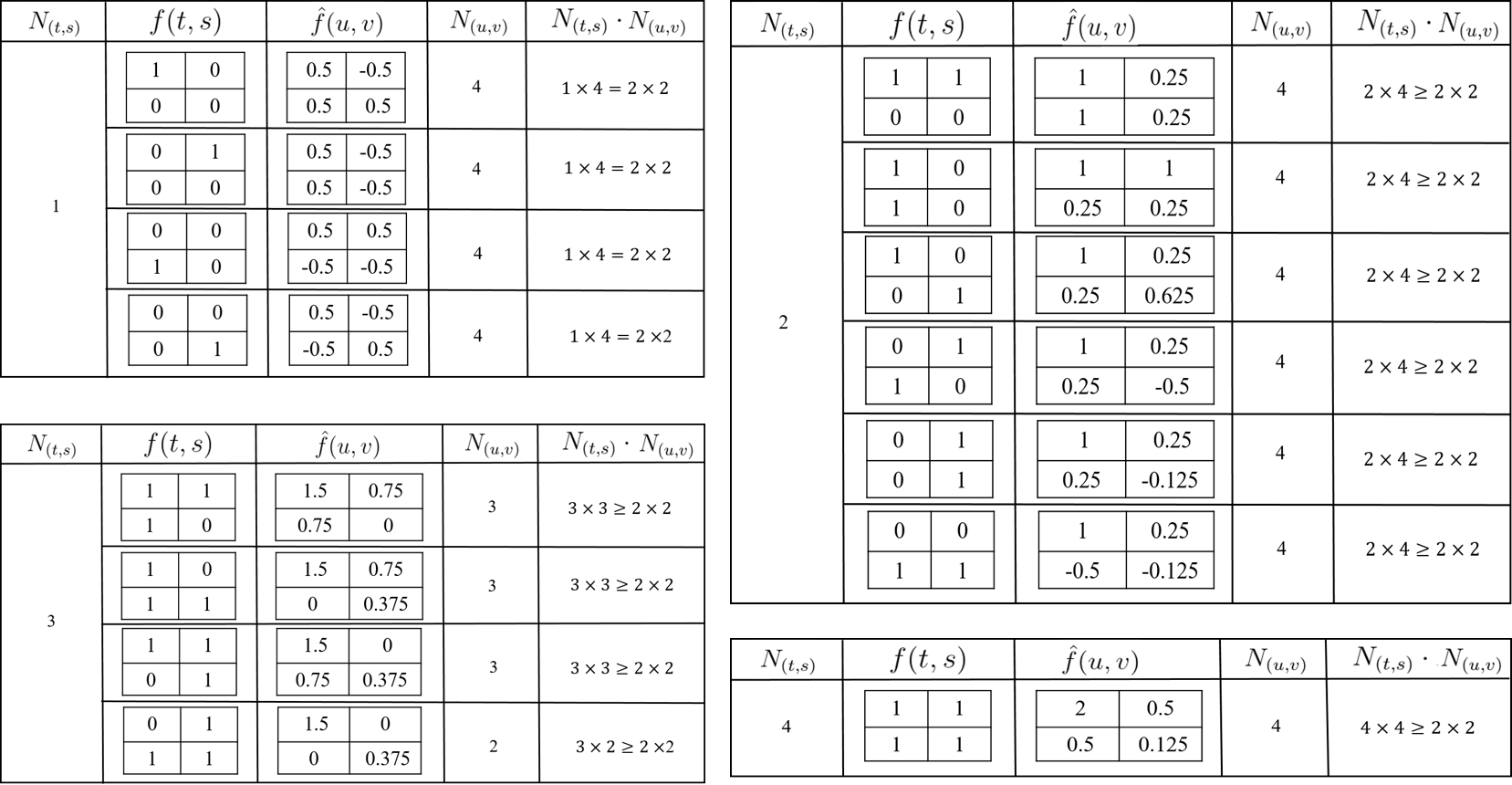}
 \caption{The cases of $N_{(t, s)}$ and  $N_{(u, v)}$ for an example of $\{f(t, s)\}$ with $M=N=2$.
 }
\label{fig:Ex3.4}
\end{figure}

%===================================================
\begin{ex} For $M=2$, $N=3$, let
\begin{eqnarray*}
A=\left(
\begin{array}{ccc}
{f}(0, 0) & {f}(0, 1) & {f}(0, 2)\\
{f}(1, 0) & {f}(1, 1) & {f}(1, 2)
\end{array}
\right).
\end{eqnarray*}

We have
\begin{eqnarray*}
&&\hat{A}\\
&=&\frac{1}{\sqrt{2}}\left(
\begin{array}{cc}
1 & 1\\
1 & e^{-\pi\i}
\end{array}
\right)
\left(
\begin{array}{ccc}
{f}(0, 0) & {f}(0, 1) & {f}(0, 2)\\
{f}(1, 0) & {f}(1, 1) & {f}(1, 2)
\end{array}
\right)
\frac{1}{\sqrt{3}}\left(
\begin{array}{ccc}
1 & 1 & 1\\
1 & e^{-\frac{2\pi \j}{3}}& e^{-\frac{4\pi \j}{3}}\\
1 & e^{-\frac{4\pi \j}{3}}& e^{-\frac{8\pi \j}{3}}
\end{array}
\right)\\
&=&\frac{1}{\sqrt{6}}\left(
\begin{array}{cc}
1 & 1\\
1 & e^{-\pi\i}
\end{array}
\right)
\left(
\begin{array}{ccc}
{f}(0, 0) & {f}(0, 1) & {f}(0, 2)\\
{f}(1, 0) & {f}(1, 1) & {f}(1, 2)
\end{array}
\right)
\left(
\begin{array}{ccc}
1 & 1 & 1\\
1 & e^{-\frac{2\pi \j}{3}}& e^{-\frac{4\pi \j}{3}}\\
1 & e^{-\frac{4\pi \j}{3}}& e^{-\frac{2\pi \j}{3}}
\end{array}
\right)\\
&=&
\left(
\begin{array}{ccc}
\hat{f}(0, 0) & \hat{f}(0, 1) & \hat{f}(0, 2)\\
\hat{f}(1, 0) & \hat{f}(1, 1) & \hat{f}(1, 2)
\end{array}
\right).
\end{eqnarray*}

A straightforward computation gives
\begin{eqnarray*}
\hat{f}(0, 0)&=&\frac{1}{\sqrt{6}}[f(0, 0)+f(1, 0)+f(0, 1)+f(1, 1)+f(0, 2)+f(1, 2)]\\
\hat{f}(0, 1)&=&\frac{1}{\sqrt{6}}[f(0, 0)+f(1, 0)+(f(0, 1)+f(1, 1))e^{-\frac{2\pi \j}{3}}+(f(0, 2)+f(1, 2))e^{-\frac{4\pi \j}{3}}]\\
\hat{f}(0, 2)&=&\frac{1}{\sqrt{6}}[f(0, 0)+f(1, 0)+(f(0, 1)+f(1, 1))e^{-\frac{4\pi \j}{3}}+(f(0, 2)+f(1, 2))e^{-\frac{2\pi \j}{3}}]\\
\hat{f}(1, 0)&=&\frac{1}{\sqrt{6}}[f(0, 0)+f(0, 1)+f(0, 2) + e^{-\pi \i }(f(1, 0)+f(1, 1)+f(1, 2))]\\
\hat{f}(1, 1)&=&\frac{1}{\sqrt{6}}[f(0, 0)+e^{-\pi\i}f(1, 0)+f(0, 1)e^{-\frac{2\pi \j}{3}}+e^{-\pi \i }f(1, 1)e^{-\frac{2\pi \j}{3}}\\
&+& f(0, 2)e^{-\frac{4\pi \j}{3}}+e^{-\pi \i }f(1, 2)e^{-\frac{4\pi \j}{3}}]\\
\hat{f}(1, 2)&=&\frac{1}{\sqrt{6}}[f(0, 0)+e^{-\pi\i}f(1, 0)+f(0, 1)e^{-\frac{4\pi \j}{3}}+e^{-\pi \i }f(1, 1)e^{-\frac{4\pi \j}{3}}\\
&+& f(0, 2)e^{-\frac{2\pi \j}{3}}+e^{-\pi \i }f(1, 2)e^{-\frac{2\pi \j}{3}}].
\end{eqnarray*}

\begin{figure}[!h]
  \centering
   \includegraphics[height=6.5cm]{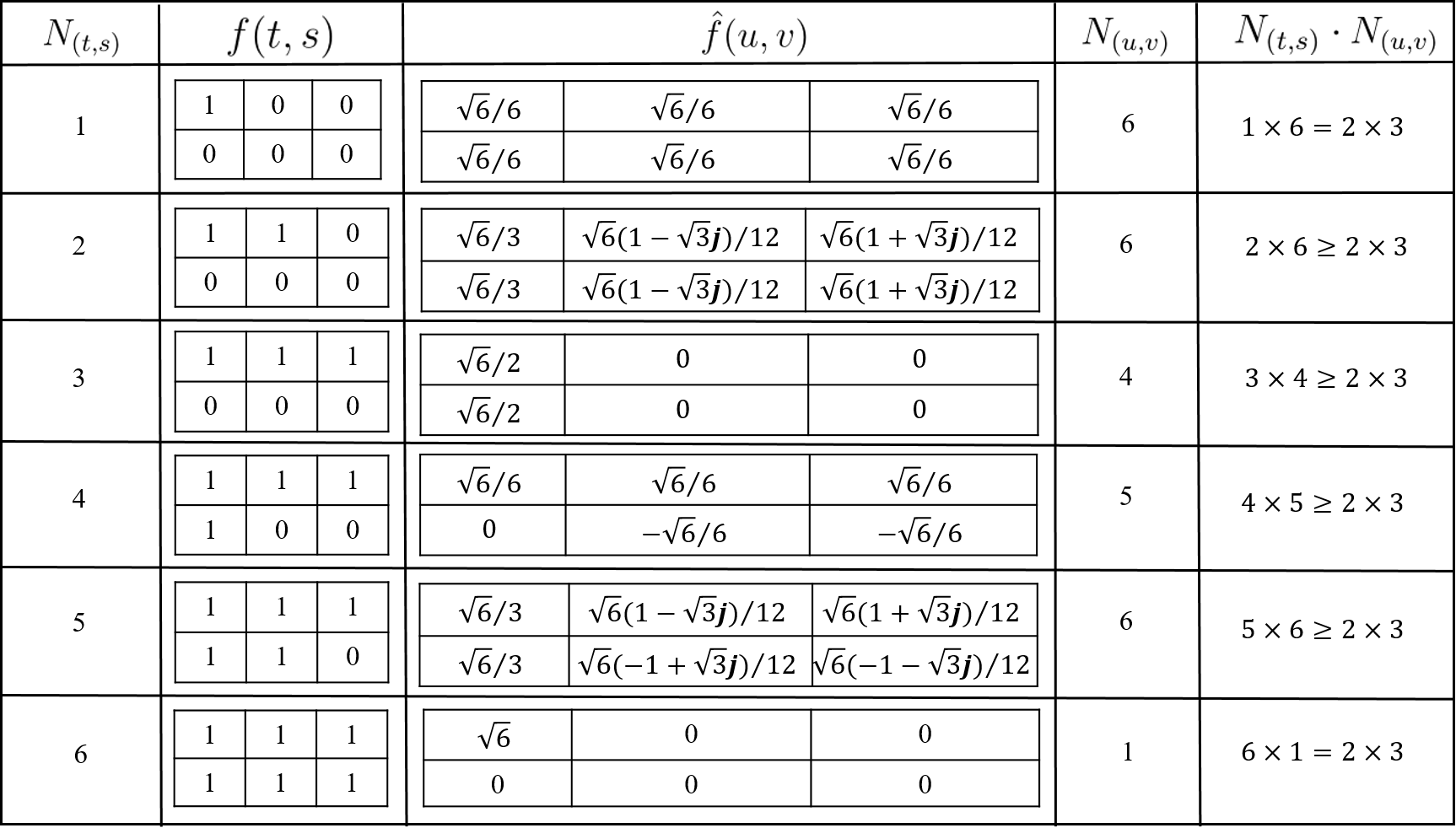}
 \caption{The cases of $N_{(t, s)}$ and  $N_{(u, v)}$ for an example of $f(t, s)$ with $M=2$ and $N=3$.
 }
\label{fig:Ex3.5}
\end{figure}

Clearly,
\begin{itemize}
\item When $\{f(t, s)\}$ has only one nonzero point, then $\{\hat{f}(u, v)\}$ has 6 nonzero points.

\item When $\{f(t, s)\}$ has two or three nonzero points, then $\{\hat{f}(u, v)\}$ has at least 3 nonzero points.

\item When $\{f(t, s)\}$ has four nonzero points, then $\{\hat{f}(u, v)\}$ has at least 2 nonzero points.

\item When $\{f(t, s)\}$ has five nonzero points, then $\{\hat{f}(u, v)\}$ has at least 2 nonzero points.

\item When $\{f(t, s)\}$ has six nonzero points, then $\{\hat{f}(u, v)\}$ has at lease 1 nonzero point.\end{itemize}

Therefore, we have $N_{(t, s)}\cdot N_{(u, v)}\geq 6.$ It also demonstrates the discrete uncertainty principle.
Figure \ref{fig:Ex3.5} shows an example for $\{f(t, s)\}$ with $M=2$, $N=3$.

\end{ex}

%======================================
\section{Uncertainty Principle for Bandlimited Signal Recovery% of a Wide-band Signal% from Narrow-band Measurements
}\label{S4}

Donoho and Stark in \cite{DS1989} gave an example where the
discrete-time uncertainty principle (\ref{yk5eq111}) shows something
unexpected is possible. That is the recovery of a $\lq\lq$ sparse " wide-band signal from narrow-band measurements. The discrete-time uncertainty principle suggests how sparsity helps in the recovery of missing frequencies. We derive the results in the quaternionic setting.

Suppose there is  an observed  discrete quaternion-valued signal $r$, which is a combination of an ideal $\Omega$-bandlimited signal $f$ and noise, i.e.
\begin{equation}\label{yk5eq10}
r :=P_{\Omega}f+n
\end{equation}
where $n$ denotes the noise and $P_{\Omega}$ is the operator that limits the measurements to the passband $\Omega$ of the system. Let $P_{\Omega}$ be the ideal bandpass operator
\begin{equation}\label{POmega}
P_{\Omega} f :=\frac{1}{\sqrt{MN}}\sum_{(u, v) \in \Omega} e^{2\pi\i\frac{ ut}{M}}\hat{f}(u, v)e^{2\pi\j\frac{ vs}{N}}.
\end{equation}

If we apply the QDFT, (\ref{yk5eq10}) becomes
\begin{eqnarray}
\hat{r}=\left\{
\begin{array}{lll}
\hat{f}+\hat{n}, & (u, v)\in \Omega,\\
0,& (u, v)\in \Omega^{c}.
\end{array}
\right.
\label{recover}
\end{eqnarray}
Here, we assumed that the noise $n$ is also bandlimited and $\Omega^{c}$ denotes the set of unobserved frequencies $\R^2 \setminus \Omega$.

Let $\Lambda :=\Omega^{c}$ and $N_{(u, v)}$ denote its cardinality.
As we see, the data $\{\hat{r}(u, v): (u, v)\in \Lambda \}$ are not observed.
The receiver's aim is to reconstruct the discrete-time signal $f$
from the noisy  observed signal $r$. Although it may seem that it is impossible, the uncertainty principles says recovery is possible
provided that  $2N_{(t, s)} \cdot N_{(u, v)}<MN$ . Here $N_{(t, s)}$ and $N_{(u, v)}$ are the numbers of nonzero elements of sequences $\{f(t, s)\}$ $(t=0, 1, \cdots, M-1, s=0, 1, \cdots, N-1)$ and $\{\hat{f}(u, v)\}$ $(u=0, 1, \cdots, M-1, v=0, 1, \cdots, N-1)$, respectively. Donoho and Stark in \cite{DS1989} proved this result in the one dimensional case.
%=============================================
\begin{thm} \label{yk5th3} Suppose there is no noise in (\ref{yk5eq10}), that is $r=P_{\Omega}f$. If $f$ has only $N_{(t, s)}$ nonzero elements and if
\begin{equation}\label{yk5eq11}
2N_{(t, s)} \cdot N_{(u, v)}<MN,
\end{equation}
then $f$ can be uniquely reconstructed from $r$.
\end{thm}

%==========================================================
\begin{proof} To prove this, we first show that $f$ is the unique sequence satisfying the condition (\ref{yk5eq11}) that can generate the given data $r$.
Suppose there is another sequence $f_1$ which also generates $r$, i.e.,  $P_{\Omega}f=r=P_{\Omega}f_1$. Let $h :=f-f_1$, we have $P_{\Omega}h=0$. Since $f$ and $f_1$ have at most $N_{(t, s)}$ nonzero elements, clearly, $h$ has fewer than $N'_{(t, s)}=2N_{(t, s)}$ nonzero elements. On the other hand, $P_{\Omega}h=0$, we have $\hat{h}(u, v)=0$, for $(u, v)\in \Omega $. Therefore the DQFT of $h$ has at most $N_{(u, v)}$ nonzero elements. Then $h$ must be zero, for otherwise it would be a contradiction with the discrete-time uncertainty principle \ref{yk5th1} (Here $N'_{(t, s)} \cdot N_{(u, v)}\leq MN$). Thus $f=f_1$. It establishes the uniqueness.

To reconstruct $f$ from observed $r$, a ideal closest point algorithm could be used.
Let $N_{(t, s)}$ be given and denote $\prod$ be the
subsets $\tau$ of $\{0, 1, \cdots, MN-1\}$ having $N_{(t, s)}$ elements. For a given subsets $\tau\in \prod$, let $\tilde{f_{\tau}}$ be the sequence supported on $\tau$, which is closed to generating the observed signal $r$, i.e.
\begin{eqnarray}
\tilde{f_{\tau}} =\argmin_{\tau\in \prod} \|r-P_{\Omega}f'\|,~~s.t.~~P_{\tau}f'=f.
\end{eqnarray}
For a fixed $\tau\in\prod$, we merely have to find that \begin{eqnarray}
\tilde{f}=\argmin_{\tilde{f_{\tau}},\tau\in\prod}
\{\|r-P_{\Omega}\tilde{f}_{\tau}\|\}.
\end{eqnarray}
%Here $\prod$ denote the subsets $\tau$ of $\{0, 1, \cdots, (N-1)^2\}$ having $}$ elements.
This step requires solving ${N_{(t, s)}}\choose{MN}$ sets of linear least-squares problems, each one requiring $\mathcal{O}((MN)^3)$ operations, therefore it is totally impractical for large $N_{(t, s)}$. It completes the proof and this theorem establishes uniqueness.
\end{proof}

In the following, one establishes stability in the presence of noise.
%================================
\begin{thm}
Suppose that $f$ has at most $N_{(t, s)}$ nonzero elements, with $$2N_{(t, s)} \cdot N_{(u, v)}<MN.$$ Assume that the norm of the noise is small, i.e., $\|n\|\leq \varepsilon$. If $\tilde{f}$ has at most $N_{(t, s)}$ nonzero elements and satisfies
\begin{equation}\label{ma1}
\|r-P_{\Omega}\tilde{f}\|\leq \varepsilon,
\end{equation}
then
$$\|f-\tilde{f}\|\leq \frac{2\varepsilon}{\sqrt{1-\frac{2N_{(t, s)}N_{(u, v)}}{MN}}}.$$
\end{thm}
%==========================
\begin{proof}
Let $T$ denote the support of $f-\tilde{f}$, then the cardinality of $T$ is at most $N'_{(t, s)}=2N_{(t, s)}$. Denote by $P_{T}$ the operator that timelimited a sequence on $T$. We have
\begin{equation}\label{ma2}
\|f-\tilde{f}\|^2=\|P_B(f-\tilde{f})\|^2+\|(I-P_B)(f-\tilde{f})\|^2.
\end{equation}
As a consequence of triangle inequality, the hypothesis condition $\|n\|\leq\varepsilon$ and inequality (\ref{ma1}), we have
\begin{eqnarray}\label{ma3}
\|P_T(f-\tilde{f})\|^2&=&\|P_T(f)-r+r-P_T(\tilde{f})\|^2\nonumber\\
&\leq&(\|P_T(f)-r\|+\|r-P_T(\tilde{f})\|)^2\nonumber\\
&=&(\|n\|+\varepsilon)^2\nonumber\\
&\leq& 4\varepsilon^2.
\end{eqnarray}
Let $P_W=I-P_T$. Then the second term of (\ref{ma2}) is
\begin{eqnarray}\label{ma4}
\|P_W(f-\tilde{f})\|^2&=&\|P_WP_T(f-\tilde{f})\|^2\nonumber\\
&\leq&\|P_WP_T\|^2\|f-\tilde{f}\|^2\nonumber\\
&\leq& \frac{2N_{(t, s)}N_{(u, v)}}{MN}\|f-\tilde{f}\|^2.
\end{eqnarray}
Combining (\ref{ma2})-(\ref{ma4}), we obtain
$$\|f-\tilde{f}\|^2\leq\frac{4\varepsilon^2}{1-\frac{2N_{(t, s)}N_{(u, v)}}{MN}}.$$ This completes the proof.
\end{proof}

\begin{ex}
For an image with size $M=N=400$, we can also find the uncertainty principle
in the image recovery processing in Fig. \ref{fig:sec4}.
For the original Lena $a1$ and different bandlimited Lena $b1$ and $c1$ in Fig. \ref{fig:sec4},
they are recovered with different numbers of $N_{(t, s)}$ and  $N_{(u, v)}$.
The results show that for different numbers of $N_{(t, s)}$ and  $N_{(u, v)}$ the PSNR and SSIM are different.
For an image with the numbers of $N_{(t, s)}$ and $N_{(u, v)}$ are smaller, the recovery results are worse.
From the Table \ref{tab1}, we can show the recovery results in data,
 the bigger of  PSNR and SSIM, the quality of images are better.
\begin{table}[!h]
\footnotesize
\caption{The number of $N_{(t, s)}$ and  $N_{(u, v)}$ in the Fig. \ref{fig:sec4}.}
\begin{center}
\begin{tabular}
{c|c|c|c|c|c}
\hline\cline{1-6}
Image size  &$N_{(t, s)}$   &$N_{(u, v)}$ &Uncertainty Principle  &PSNR  &SSIM
\\\hline
(a) $400\times400$ &$160000$  &$82057$  &$160000\times82057\geq400\times400$  &$33.4804$  &$0.9985$
\\
(b) $400\times400$ &$79841$  &$46396$  &$79841\times46396\geq400\times400$  &$26.4401$  &$0.4592$
\\
(c) $400\times400$ &$79800$  &$65990$  &$79800\times65990\geq400\times400$  &$27.7791$  &$0.7151$
\\ \hline\cline{1-6}
\end{tabular}
\end{center}
\label{tab1}
\end{table}

%\begin{figure}[ih]
%  \centering
%   \includegraphics[height=8cm]
%    \caption{The cases of $N_{(t, s)}$ and  $N_{(u, v)}$ for an example of $\{f(t, s)\}$ with $M=N=400$.
% }
%\label{fig:sec4}
%\end{figure}
%%%%%%%%%%%%%%%%%%%%%%%%%%%%%%%%%%%%%%%%%%%%%%%%%%%%%%%%%%%%%%%%%%%%%%%%%%%%%%%%%
\begin{figure}
  \centering
   \includegraphics[height=7cm]{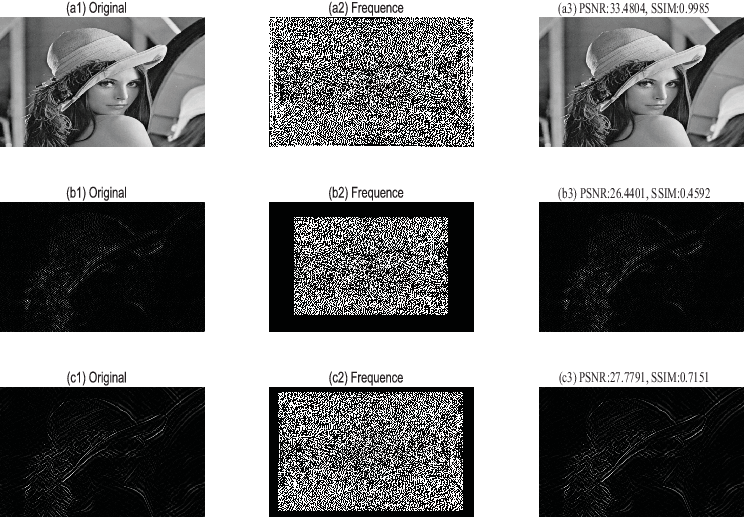}
 \caption{The cases of $N_{(t, s)}$ and  $N_{(u, v)}$ for an example of $\{f(t, s)\}$ with $M=N=2$.
 }
\label{fig:sec4}
\end{figure}
%========================================

\end{ex}

%%%%%%%%%%%%%%%%%%%%%%%%%%%%%%%%%%%%%%%%%%%%%%%%%%%%%%%%%%%%%%%%%%%%%%%%%%%%%%%%%

%========================================
\section*{Acknowledgment}
The authors acknowledge financial support from Science and Technology Program of Guangzhou, China (No. 74120-42050001), the National Natural Science Foundation of China under Grant (No. 61806027), Macao Science and Technology Development Fund (FDCT/031/2016/A1 and FDCT/085/2018/A2)

%-------------------------------------------

% ------------------------------------------------------------------------

\begin{thebibliography}{12}

\bibitem{BLM2017} A. S. Bandeira, M. E. Lewis, D. G. Mixon, \textit{Discrete Uncertainty Principle and Sparse Siganl Processing}, Journal of Fourier Analysis and Applications, 2017.
DOI 10.1007/s00041-017-9550-x

\bibitem{BTN2007} E. Bayro-Corrochano, N. Trujillo, M. Naranjo, \textit{Quaternion Fourier descriptors for preprocessing
and recognition of spoken words using images of spatiotemporal
representations}, Journal of Mathematical Imaging and Vision 28(2):
179-190 (2007).

\bibitem{BLC2003} P. Bas, N. Le Bihan, J. M. Chassery, \textit{Color image watermarking using quaternion Fourier transform}
in Proceedings of the IEEE International Conference on Acoustics
Speech and Signal and Signal Processing, ICASSP, Hong-kong, 521-524
(2003).

\bibitem{B1999} T. B$\ddot{u}$low, \textit{Hypercomplex spectral signal representations for the processing and analysis
of images}, Ph.D. Thesis, Institut f$\ddot{u}$r Informatik und
Praktische Mathematik, University of Kiel, Germany, (1999).

\bibitem{CT2006} E. J. Candes, T. Tao, \textit{Near-optimal signal recovery from random projections: Universal encoding strategies}, IEEE Trans. Inf. Theory, 52: 5406-5425 (2006).

\bibitem{DS1989} D. L. Donoho and P. B. Stark, \textit{Uncertainty principles and signal
recovery}, SIAM Journal on Applied Mathematics,  49(3): 906-931 (1989).

\bibitem{E1993} T. A. Ell, \textit{Quaternion-fourier transfotms for analysis of two-dimensional linear time-invariant
partial differential systems}, in: Proceeding of the 32nd Conference
on Decision and Control, San Antonio, Texas, 1830-1841 (1993).

\bibitem{H1927} W. Heisenberg, \textit{$\ddot{U}$ber den anschaulichen Inhalt der quantentheoretischen Kinematik und Mechanik}, Z. Phys.(in German) 43: 172-198 (1927).

\bibitem{HK2018} X. X Hu, K. I. Kou, \textit{Phase based edge detection algorithms}, Mathematical Methods in the Applied Sciences 41(11): 4148-4169 (2018).


\bibitem{K2017} K. I. Kou, Y. Yang and C.M. Zou, \textit{Uncertainty principle for measurable sets and signal recovery in quaternion domains}, Mathematical Methods in the Applied Sciences 40(11): 3892-3900 (2017).

\bibitem{KLMZ2017} K. I. Kou, M. S. Liu, J. P. Morais, C. M. Zou, \textit{Envelope detection using generalized analytic signal in 2D QLCT domains}, Multidimensional Systems and Signal Processing 28(4): 1343-1366 (2017).

\bibitem{S1979} A. Sudbery. \textit{Quaternionic analysis}. Math. Proc. Cambridge Phil. Soc. 85: 199-225 (1979).

\bibitem{SE2007} S. J. Sangwine, T. A. Ell, \textit{Hypercomplex Fourier transforms of color images}, IEEE Transactions on Image Processing 16(1): 22-35 (2007).

\bibitem{S1997} S. J. Sangwine. \textit{The discrete quaternion Fourier transform}, 1997 Sixth International Conference on Image Processing and Its Applications, Dublin, Ireland, 1997, pp. 790-793 vol.2.
doi: 10.1049/cp:19971004

\bibitem{T2008} J. A. Tropp, \textit{On the linear independence of spikes and sines}, J. Fourier Appl. 14: 838-858 (2008).
\end{thebibliography}
\end{document}